\documentclass[12pt]{amsart}
\pdfoutput=1
\usepackage{amssymb, amsmath, amsthm}
\RequirePackage{graphicx}
\usepackage{mathrsfs}

\newcommand{\supp}{\operatorname{supp}}

\newtheorem{theorem}{Theorem}
\newtheorem{lemma}[theorem]{Lemma}
\newtheorem{proposition}[theorem]{Proposition}

\newcommand{\newsection}[1] {\section{#1}\setcounter{theorem}{0}
\setcounter{equation}{0}\par\noindent}

\newcommand{\R}{{\mathbb R}}

\newcommand{\g}{{\rm g}}

\newcommand{\C}{{\mathcal C}}
\newcommand{\hf}{\frac 12}

\begin{document}

\title[Bounds on quasimodes for semiclassical Schr\"odinger operators]
{Pointwise bounds on quasimodes of semiclassical Schr\"odinger
  operators in dimension two }
\thanks{HS is supported in part by NSF grant
  DMS-1161283 and MZ is supported in part by NSF grant DMS-1201417.}
\author{Hart F. Smith and Maciej Zworski}
\address{Department of Mathematics, University of Washington, Seattle, WA 98195}
\email{hart@math.washington.edu}
\address{Department of Mathematics, University of California, Berkeley, CA 94720}
\email{zworski@math.berkeley.edu}

\begin{abstract}
We prove optimal pointwise bounds on quasimodes of semiclassical
Schr\"odinger operators with arbitrary smooth real potentials in
dimension two. This end-point estimate was left open in the
general study of semiclassical $ L^p $ bounds conducted by
Koch-Tataru-Zworski \cite{KTZ}. However, we show that 
the results of \cite{KTZ} imply the two dimensional end-point estimate
by scaling and localization.
\end{abstract}

\maketitle 

\newsection{Introduction}

Let $\g_{ij}(x)$ be a positive definite Riemannian metric on $\R^2$
with the corresponding Laplace-Beltrami operator, 
\[  \Delta_\g u :=  \frac 1 { \sqrt {\bar \g} } \sum_{ i,j} \partial_{x_j
}\left(  \g^{ij}  \sqrt {\bar \g} \, \partial_{x_j} u \right) , \ \  (
\g^{ij} ) := ( \g_{ij} )^{-1} , \ \ \bar \g := \det ( \g_{ij} ) , \]
and let $V \in C^\infty ( \R^2  ) $ be real valued.  We prove the 
following general bound which was already established (under an
additional necessary condition) in higher dimensions in \cite{KTZ}, but
which was open in dimension two: 

\medskip
\noindent
{\bf Theorem.} 
{\em Suppose that $h\le 1$, and that $u\in H^2_c(\R^2)$ with $\text{supp}(u) \subseteq K \Subset \R^2$.
Suppose that $u$ satisfies
\begin{equation}\label{eqn:semiclass}
\bigl\| - h^2\Delta_\g u+Vu\|_{L^2}\le h\,,\qquad \|u\|_{L^2}\le 1\,.
\end{equation}
Then
\begin{equation}\label{est:semiclass}
\|u\|_{L^\infty}\le C\,h^{-\frac 12}\,,
\end{equation}
where the constant $C$ depends only on $\g$, $V$, and $K$.
}

\medskip

A function $ u $ satisfying \eqref{eqn:semiclass} is sometimes called
a weak quasimode. It is a local object in the sense that if 
$ u $ is a weak quasimode then $ \psi u $, $ \psi \in C_{c}^\infty (
\R^2 )  $ is also one. The localization is also valid in phase
space: for instance if  $ \chi \in C_c^\infty ( \R^2 \times \R^2 ) $
then  $ \chi^w ( x, h D) u $ is also a weak quasimode -- 
see \cite[Chapter 7]{DiSj} or \cite[Chapter 4]{e-z} for the 
review of the Weyl quantization $ \chi \mapsto \chi^w $.

If $  \liminf_{|x| \to \infty } V > 0 $, then $ - h^2 \Delta + V $
(defined on $ C_c ^\infty ( \R^2 ) $) is 
essentially self-adjoint and the spectrum of
$ - h^2 \Delta + V $ is discrete in a neighbourhood of $ 0$ --
see for instance \cite[Chapter 4]{DiSj}.
In this case weak quasimodes arise as {\em spectral clusters}:
\begin{equation}
\label{eq:specc}   w  = \sum_{ | E_j | \leq C h } c_j w_j , \ \  ( - h^2 \Delta + V
) w_j = E_j w_j ,  \ \ \langle w_j , w_k \rangle_{L^2} = \delta_{jk} ,
\ \ \sum_{ j} |c_j|^2 \leq 1 .\end{equation}
Then $ u = \chi w $, for any $ \chi \in C_c^\infty ( \R^2 ) $, 
is a weak quasimode in the sense of \eqref{eqn:semiclass}. 
Since $ V(x ) \geq c > 0 $  for $ |x| \geq R $,  Agmon
estimates (see for instance \cite[Chapter 6]{DiSj}) 
and Sobolev embedding show that $ |u ( x ) | \leq e^{ c/h } $ for
$ |x| \geq R$. Hence we get global bounds
\[  | w ( x ) | \leq C h^{-\frac12} , \ \ x \in  \R^2 .\]
It should be stressed however that 
a weak quasimode is a more general notion than a spectral cluster.

The result also holds when $ \R^2 $ is replaced by 
a two dimensional manifold and, as in the example
above, gives global bounds on spectral clusters 
\eqref{eq:specc} when the manifold is compact. 
If $V<0$ this 
is also a by-product of the bound of Avakumovic-Levitan-H\"ormander 
on the spectral function -- see \cite{Sogge}, and for a simple
proof of a semiclassical generalization see \cite[\S 3]{KTZ} or \cite[\S 7.4]{e-z}.

In higher dimensions the theorem requires an additional phase
space localization assumption and is a special case of 
\cite[Theorem 6]{KTZ}: Suppose $ p ( x, \xi ) $ is a function 
on $ \R^n \times \R^n $ satisfying
$ \partial_x^\alpha \partial_\xi^\beta p(x,\xi) = \mathcal O ( \langle \xi \rangle^m) $
 for some $m$. Suppose also that $ K \Subset \R^n \times \R^n $, and that for $ ( x  , \xi )  \in K $
\[   p ( x, \xi ) =  0 , \ \  \partial_\xi p ( x , \xi ) = 0  \
\Longrightarrow \ \partial_\xi^2 p ( x , \xi ) \ \text{ is
  nondegenerate. } \]
Then for $ u( h )  $ such that 
\begin{equation}
\label{eq:uK} 
 u ( h ) = \chi^w ( x, h D)  u +
{\mathcal O}_{\mathscr S}  ( h^\infty ) \, , \ \ \  \supp \chi \subset
K \,, \end{equation} 
we have 
\begin{equation}
\label{eq:KT0}  \| u ( h )  \|_\infty \leq C h^{ - \frac{n-1} 2 } \left( \| u ( h
  ) \|_{ L^2} + \frac 1 h \| p^w ( x , h D ) u \|_{ L^2 } \right) , \
\ n \geq 3.  
\end{equation}
When $ n = 2 $ the bound holds with $  (\log ( 1/h ) / h )^{\frac12} $,
which is optimal in general if $ \partial^2_\xi p $ is not positive
definite -- see \cite[\S 3, \S 6]{KTZ} and \S 3
below for examples.

A small bonus in dimension two is the fact that the frequency
localization condition \eqref{eq:uK} needed for \eqref{eq:KT0} is not 
necessary -- see \eqref{eq:chiD} below.

The proof of the theorem is reduced to a local result presented in
Proposition \ref{theorem2}. That result follows in turn from a rescaling
argument involving several cases, some of
which use the most technically involved result of \cite{KTZ}: if
$ p ( x, \xi ) = \sum_{ i, j} \g^{ij} \xi_i \xi_j  + V ( x )$ and  for
$ ( x, \xi ) \in K $
\[  p ( x, \xi ) = 0,  \ \ \partial_\xi p  ( x, \xi ) = 0 \ 
\Longrightarrow \ d_x p ( x, \xi ) = d V ( x ) \neq 0 , \]
then for $ u ( h) $ satisfying \eqref{eq:uK}, the bound \eqref{eq:KT0} holds
for any $ n \geq 2 $ -- see \cite[Corollary 1]{KTZ} for the original results and
Propositions \ref{proposition1} and \ref{proposition2} for rescaled
versions used in our proof.  We do not know of any simpler way to
obtain \eqref{eqn:semiclass}.

\newsection{Proof of Theorem}
By compactness of $K\supseteq\text{support}(u)$, it suffices to prove uniform $L^\infty$ bounds on $u$ over each
small ball intersecting $K$, where in our case the diameter of the ball can be taken to depend only on $\C^N$ estimates for $\g$ and $V$ over a unit sized neighborhood of $K$, for some large $N$. Without loss of generality we consider a ball centered at the origin in $\R^2$.
Let
$$
B=\{x\in\R^2\,:\,|x|<1\}\,,\qquad B^*=\{x\in\R^2\,:\,|x|<2\}\,.
$$
After a linear change of coordinates, we may assume that
\begin{equation}\label{cond1}
\g^{ij}(0)=\delta^{ij}\,.
\end{equation}
Next, by replacing $V(x)$ by $cV(cx)$ and $\g^{ij}(x)$ by $\g^{ij}(cx)$, for some constant
$c\le 1$ depending on the $\C^2$ norm of $\g$ and $V$ over a unit neighborhood of $K$, 
we may assume that
\begin{equation}\label{cond2}
\sup_{x\in B^*}|V(x)|+|dV(x)|\le 2\,,\quad \sup_{x\in B^*}|d^2 V(x)|+\sum_{i,j=1}^2|d\g^{ij}(x)|\le .01\,.
\end{equation}
This has the effect of multiplying $h$ by a constant in the equation \eqref{eqn:semiclass},
which can be absorbed into the constant $C$ in \eqref{est:semiclass}.

In general, we let
\begin{equation}\label{cond3}
C_N=\sup_{x\in B^*}\sup_{|\alpha|\le N}
\Bigl(\,|\partial^\alpha V(x)|+\sum_{i,j=1}^2|\partial^\alpha\g^{ij}(x)|\,\Bigr)\,,
\end{equation}
and will deduce the main theorem as a corollary of the following 

\begin{proposition}\label{theorem2}
Suppose $h\le 1$, that $\g,V$ satisfy \eqref{cond1} and \eqref{cond2}, and that
$u$ satisfies
\begin{equation}\label{eqn:semiclass'}
\bigl\| -h^2\Delta_\g u+Vu\|_{L^2(B^*)}\le h\,,\qquad \|u\|_{L^2(B^*)}\le 1\,.
\end{equation}
Then
\begin{equation}\label{est:semiclass'}
\|u\|_{L^\infty(B)}\le C\,h^{-\frac 12}\,,
\end{equation}
where the constant $C$ depends only on $C_N$ in \eqref{cond3} for some fixed $N$.
\end{proposition}

We start the proof of Proposition \ref{theorem2} by recording the following two propositions, which are consequences of \cite[Corollary 1]{KTZ}. 

\begin{proposition}\label{proposition1}
Suppose that \eqref{cond1}-\eqref{cond2} hold, and that $\frac 12 \le |V(x)|\le 2$ for $|x|\le 2$. If the following holds, and $h\le 1$,
$$
\bigl\| -h^2\Delta_\g u+Vu\|_{L^2(B^*)}\le h\,,\qquad \|u\|_{L^2(B^*)}\le 1\,,
$$
then $\|u\|_{L^\infty(B)}\le C\,h^{-\frac 12}$, where $C$ depends only on $C_N$ in \eqref{cond3} for some fixed $N$.
\end{proposition}

\begin{proposition}\label{proposition2}
Suppose that \eqref{cond1}-\eqref{cond2} hold, and that $V(0)=0$ and $|dV(0)|=1$. 
If the following holds, and $h\le 1$,
$$
\bigl\| -h^2\Delta_\g u+Vu\|_{L^2(B^*)}\le h\,,\qquad \|u\|_{L^2(B^*)}\le 1\,,
$$
then $\|u\|_{L^\infty(B)}\le C\,h^{-\frac 12}$, where $C$ depends only on $C_N$ in 
\eqref{cond3} for some fixed $N$.

\end{proposition}

To see that these follow from \cite[Corollary 1]{KTZ}, we first note that
in Proposition \ref{proposition2} above, since $|d^2V|\le .01$, we have $.98\le |dV(x)|\le 1.02$ for 
$|x|\le 2$, so since $\g$ is positive definite the conditions on $\g$ and $V$ in \cite[Corollary 1]{KTZ} are met.
We remark that the condition $V(0)=0$ guarantees the the zero set of $V$ is a nearly-flat curve through the origin, although this is not strictly needed
to apply the results of \cite{KTZ}. That the constants in the estimates of \cite{KTZ} depend only on the above bounds on $\g$ and $V$ follow from their proofs.

Next, the estimates above can be localized, as remarked before, so we may assume that $u$ is compactly supported in $|x|<\frac 32$, after which we may extend $\g$ and $V$ globally without
affecting the application of \cite[Corollary 1]{KTZ}.
Indeed, in both propositions above the
assumptions imply $\|du\|_{L^2(|x|<{3/2})}\lesssim h^{-1}$, so that one may cut off $u$ by a smooth
function which is supported in $|x|<\frac 32$ and equals 1 for $|x|<1$.

Finally, the condition (1.4) of \cite{KTZ} that 
$u-\chi(hD)u={\mathcal O}_{\mathscr S}  ( h^\infty )$ for some
$\chi\in \C_c^\infty$ is not needed for the $L^\infty$ results of that paper to hold in dimension two. To see this, we note that since $|V|<2$ and $|\g^{ij}(x)-\delta_{ij}|\le .02$
on the ball $|x|<2$, then if $u$ is supported in $|x|<\frac 32$ and $\chi(\xi)=1$ for $|\xi|<4$, then
$$
\|(hD)^2(u-\chi(hD)u)\|_{L^2} = {\mathcal O}(h)\,.
$$
This follows by the semiclassical pseudodifferential calculus (see \cite[Theorem 4.29]{e-z}),
since for $ \chi_0 \in C^\infty_c  ( \R^2  ) $ with $\supp \chi_0 \subset B^*$, 
$ \chi_0 ( x )(1 - \chi ( \xi ) ) |\xi|^2/( |\xi|^2 + V( x)  )   \in 
S ( \R^2 \times \R^2 ) $.

Hence, writing $ \hat u ( \xi ) $ for the standard Fourier transform of $u$,
\begin{equation}
\label{eq:chiD}
\begin{split}  \|u-\chi(hD)u\|_{L^\infty} & \le  \frac 1 { (2 \pi)^2} \int_{
    \R^2} | 1 - \chi ( h \xi ) | | \hat u ( \xi ) | \,d \xi  \\
& \leq 
C \int  | h \xi |^2 | 1 - \chi ( h \xi ) | | \hat u ( \xi ) | 
( 1 +|h\xi|^2 )^{-1} \, d \xi  \\
& \leq C \|(hD)^2(u-\chi(hD)u) \|_{L^2} \left( \int_{ \R^2 } ( 1 + | h \xi |^2
)^{-2} \, d \xi \right)^{\frac12} \\ 
& \leq C h \, h^{-1} = C \,,
\end{split}
\end{equation}
an even better estimate than required.

We supplement Propositions \ref{proposition1} and \ref{proposition2} with the following two lemmas.

\begin{lemma}\label{lemma3}
Suppose that \eqref{cond1}-\eqref{cond2} hold, and
that $|V(x)|\le 99\,h$ for $|x|\le 2h^{\frac 12}$.
If the following holds, and $h\le 1$,
$$
\bigl\| -h^2\Delta_\g u+Vu\|_{L^2(|x|<2h^{1/2})}\le h\,,\qquad 
\|u\|_{L^2(|x|<2h^{1/2})}\le 1\,,
$$
then $\|u\|_{L^\infty(|x|<h^{1/2})}\le C\,h^{-\frac 12}$, where $C$ depends only on 
$C_N$ in \eqref{cond3} for some fixed $N$.
\end{lemma}
\begin{proof}
Consider the function $\tilde u(x)=h^{\frac 12} u(h^{\frac 12}x)$, and ${\tilde \g}^{ij}(x)=\g^{ij}(h^{\frac 12}x)$. Then, since $\|Vu\|_{L^2(|x|<2h^{1/2})}\le 99h$, we have
$$
\|\Delta_{\tilde \g}\tilde u\|_{L^2(|x|<2)}\le 100\,,\qquad \|\tilde u\|_{L^2(|x|<2)}\le 1\,.
$$
Since the spatial dimension equals 2, interior Sobolev estimates yield
$\|\tilde u\|_{L^\infty(|x|<1)}\le C$, where we note that the conditions \eqref{cond1} and \eqref{cond2} hold for $\tilde\g$ since $h^{\frac 12}\le 1$.
\end{proof}

\begin{lemma}\label{lemma4}
Suppose that \eqref{cond1}-\eqref{cond2} hold, and
that $\hf c\le |V(x)|\le 2c$ for $|x|\le 2c^{\frac 12}$.
If the following holds, and $h\le c\le 1$,
$$
\bigl\| -h^2\Delta_\g u+Vu\|_{L^2(|x|<2c^{1/2})}\le h\,,\qquad 
\|u\|_{L^2(|x|<2c^{1/2})}\le 1\,,
$$
then $\|u\|_{L^\infty(|x|<c^{1/2})}\le C\,h^{-\frac 12}$, where $C$ depends only $C_N$ in \eqref{cond3} for some fixed $N$.
\end{lemma}
\begin{proof}
Let $\tilde u(x)=c^{\frac 12} u(c^{\frac 12}x)$, ${\tilde \g}^{ij}(x)=\g^{ij}(c^{\frac 12}x)$, and 
$\tilde V(x)=c^{-1}V(c^{\frac 12}x)$. Note that the assumptions on $V(x)$ in the statement and in \eqref{cond2} imply that
$|dV(x)|\le c^{\frac 12}$ for $|x|<2 c^{1/2}$, so that $\tilde V$ satisfies \eqref{cond2}, and the constants $C_N$ in \eqref{cond3} can only decrease for $c\le 1$. Then with
$\tilde h=c^{-1}h\le 1$, 
$$
\|-\tilde h^2\Delta_{\tilde\g} \tilde u+\tilde V\tilde u\|_{L^2(|x|<2)}\le \tilde h\,,\qquad \|\tilde u\|_{L^2(|x|<2)}\le 1\,.
$$
By Proposition \ref{proposition1}, we have 
$\|\tilde u\|_{L^\infty(|x|<1)}\le C{\tilde h}^{-\frac 12}$, giving the desired result.
\end{proof}

\begin{proof}[Proof of Proposition \ref{theorem2}] 
It suffices to prove that for each $|x_0|<1$ there is some $\hf\ge r>0$ so that
$\|u\|_{L^\infty(|x-x_0|<r)}\le C\, h^{-\frac 12}$, with a global constant $C$. Without loss of generality we take $x_0=0$.

We will split consideration up into four cases, depending on the relative size of $|V(0)|$ and $|dV(0)|$. Since for $h$ bounded away from $0$ the result follows by elliptic estimates, we will assume $h\le\frac 14$ so that $h^{\hf}$ below is at most $\hf$.

\noindent{\bf Case 1:} $|V(0)|\le h\,,\;|dV(0)|\le 8 h^{\frac 12}$. Since $|d^2V(x)|\le .01$, then Lemma \ref{lemma3} applies to give the result with $r=h^{\frac 12}$.

\medskip

\noindent{\bf Case 2:} $|V(0)|\le h\,,\;|dV(0)|\ge 8h^{\frac 12}$. Since we may add a constant of size $h$ to $V$ without affecting \eqref{eqn:semiclass'}, we may assume $V(0)=0$. By rotating we may then assume 
$$V(x)=\beta x_1+f_{ij}(x)x_i x_j\,,$$
where $\beta=|dV(0)|\ge 8h^{\frac 12}$. Dividing $V$ by 4 if necessary we may assume $\beta\le \hf$.
Let $\tilde u=\beta u(\beta x)$, 
$\tilde\g^{ij}(x)=\g^{ij}(\beta x)$,
and
$$
\tilde V(x)=\beta^{-2}V(\beta x)=x_1+f_{ij}(\beta x)x_i x_j\,.
$$
With $\tilde h=\beta^{-2}h<1$ we have
$$
\|-\tilde h^2\Delta_{\tilde\g}\tilde u +\tilde V\tilde u\|_{L^2(|x|<2)}\le \tilde h\,,
\qquad \|\tilde u\|_{L^2(|x|<2)}\le 1\,.
$$
Proposition \ref{proposition2} applies, since $\tilde g$ and $\tilde V$ satisfy \eqref{cond1}-\eqref{cond2}, and the constants $C_N$ in \eqref{cond3} for $\tilde g$ and $\tilde V$ are bounded by those for $\g$ and $V$.
Thus $\|\tilde u\|_{L^\infty(|x|<1)}\le C\tilde h^{-\frac 12}$, giving
the desired result on $u$ with $r=|dV(0)|$.

\medskip

\noindent{\bf Case 3:} $|V(0)|\ge h\,,\;|dV(0)|\le 9|V(0)|^{\frac 12}$.
In this case, with $c=|V(0)|$, it follows that $\hf c\le |V(x)|\le 2c$ for $|x|\le \frac 1{20} c^{\frac 12}$. We may apply Lemma \ref{lemma4} with $V$ replaced by $\frac 1{1600}V$ to get the desired result with $r=\frac 1{40}|V(0)|^{\frac 12}$.

\medskip

\noindent{\bf Case 4:} $|V(0)|\ge h\,,\;|dV(0)|\ge 9 |V(0)|^{\frac 12}$.
Since $|d^2V(x)|\le .01$, it follows that there is a point $x_0$ with 
$|x_0|\le \frac 18 |V(0)|^{\frac 12}$
where $V(x_0)=0$. Since $|dV(x_0)|\ge 8 |V(0)|^{\frac 12}\ge 8h^{\frac 12}$, we may translate and apply Case 2 to get $L^\infty$ bounds on $u$ over a neighborhood of radius $|dV(x_0)|$ about $x_0$. This neighborhood contains the neighborhood about $0$ of radius 
$r=.9998\,|dV(0)|$.
\end{proof}

\newsection{A counter-example for indefinite $\g$.}
In \cite[Section 5]{KTZ}, it was shown that there exist $u_h$ for which
\begin{equation}\label{KTZexample}
\|-h^2(\partial_{x_1}^2-\partial_{x_2}^2)u_h+(x_1^2-x_2^2)u_h\|_{L^2}\le h\,,\qquad
\|u_h\|_{L^2}\le 1\,,
\end{equation}
for which $\|u_h\|_{L^\infty}\approx |\log h|^\hf h^{-\hf}\,,$ showing that the assumption of
definiteness of $\g$ cannot be relaxed to non-degeneracy in the main
theorem. In \cite[Theorem 6]{KTZ} the positive result was established showing that this growth of $\|u_h\|_{L^\infty}$ for indefinite, non-degenerate $\g$ in two dimensions is in fact worst case.

The example of \cite{KTZ} was produced using harmonic oscillator eigenstates.
Here we present a different construction of such a $u_h$ with similar $L^\infty$ growth to help illustrate the role played by the degeneracy of $\g$. The idea is to produce a collection
$u_{h,j}$ of functions satisfying \eqref{KTZexample} (or equivalent), for which
$u_{h,j}(0)=h^{-\hf}$, and where $j$ runs over $\approx|\log h|$ different values. The examples will have disjoint frequency support, hence are orthogonal in $L^2$. Upon summation over $j$ the $L^2$ norm then grows as $|\log h|^\hf$, whereas the $L^\infty$ norm grows as $|\log h|\,h^{-\hf}$, yielding an example with worst case growth after normalization.

We start by considering the form $\xi_1\xi_2$ with $V=0$.
To assure that $\|h^2\partial_{x_1}\partial_{x_2} u_h\|_{L^2}\le h$, we will take the Fourier transform of $u_h$ to be contained in the set $|\xi_1\xi_2|\le 2h^{-1}$, as well as 
$|\xi|\le 2h^{-1}$ to satisfy the frequency localization condition \cite[(1.4)]{KTZ}. 
Our example is then based on the fact that one can
find $\approx|\log h|$ disjoint rectangles, each of volume $h^{-1}$, within this region, as illustrated in the diagram. Each
$u_{h,j}$ will be an appropriately scaled Schwartz function with Fourier transform localized to one of the rectangles.

\centerline{\includegraphics[width=6truein]{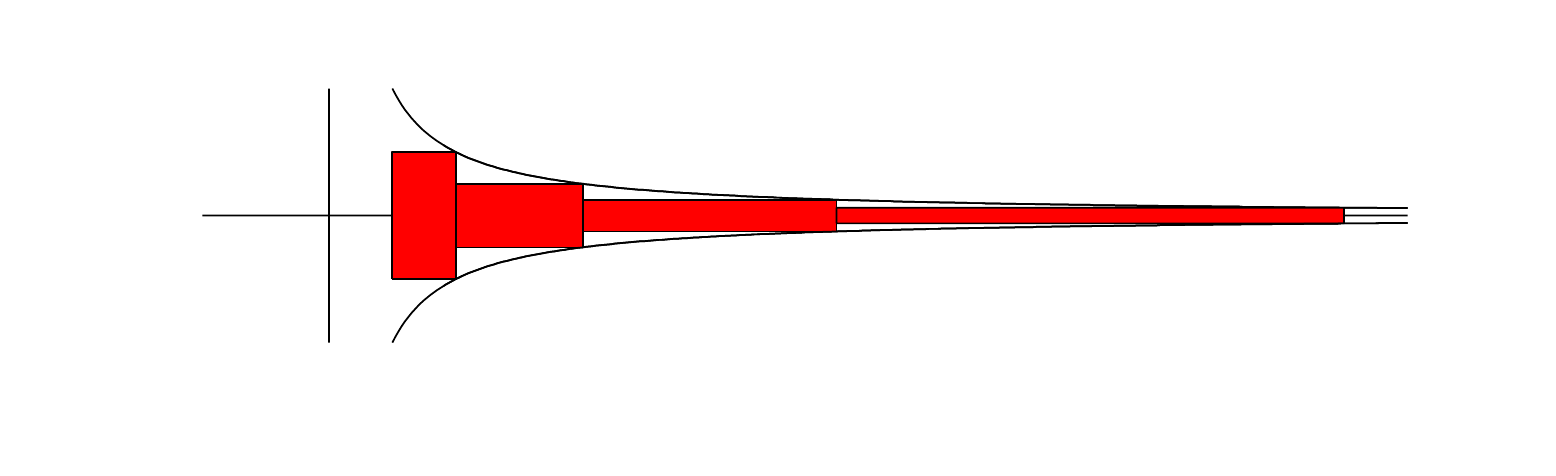}}

We now fix $\psi,\chi\in \C_c^\infty(\R)$, with $0\le \psi(x)\le 2$
and $0\le \chi(x)\le 1$,
with $\int\psi=\int\chi=1$, and where
$$
\text{supp}(\psi)\subset [1,2]\,,\qquad \text{supp}(\chi)\subset [-1,1]\,.
$$
We additionally assume $\chi(0)=1$.

Let
$$
u_{h,j}(x)=h^{\frac 12}\int e^{ix_1\xi_1+ix_2\xi_2}\chi(2^jh\,\xi_1)\psi(2^{-j}\xi_2)
\,d\xi_1\,d\xi_2
=h^{-\frac 12}\check{\chi}(2^{-j}h^{-1}x_1)\check{\psi}(2^jx_2)
\,.
$$
By the Plancherel theorem, $\|u_{h,j}\|_{L^2}\approx 1$ and 
$\|h^2D_1D_2 u_{h,j}\|_{L^2}\lesssim h\,.$
Furthermore, $u_{h,j}(0)=h^{-\frac 12}$. By disjointness of the Fourier transforms, for $i\ne j$ we have
$\langle u_{h,i},u_{h,j}\rangle =0$, and similarly 
$\langle \partial_{x_1}\partial_{x_2} u_{h,i},\partial_{x_1}\partial_{x_2} u_{h,j}\rangle =0$.

We then form 
$$
u_h(x)=|\log h|^{-\frac 12}\sum_{1\le 2^j\le h^{-1}}u_{h,j}(x)\,.
$$
Since there are $\approx|\log h|$ terms in the sum, and the terms are orthogonal in $L^2$, it follows that
$$
\|u_h\|_{L^2}\approx 1\,,\qquad \|h^2 \partial_{x_1}\partial_{x_2}u_h\|_{L^2(\R^2)}\lesssim h\,,\qquad u_h(0)\approx 
|\log h|^{\frac 12}h^{-\frac 12}\,.
$$
Although the example is not compactly supported, it is rapidly decreasing (uniformly so for $h<1$), and one may smoothly cutoff to a bounded set without changing the estimates.

We observe that for this example it also holds that 
$$
\|x_1x_2 u_h\|_{L^2}\lesssim h\,.
$$ 
Hence, $u_h$ is also a counterexample for the form
$\xi_1\xi_2\pm x_1x_2$. Rotating by $\pi/4$ gives the form $\xi_1^2-\xi_2^2\pm(x_1^2-x_2^2)$,
including in particular the form considered in \cite[Section 6]{KTZ}. 

We also observe that $x_1^2u_h$ will be $\mathcal{O}_{L^2}(h)$ if one restricts the sum in $u_h$ to $1\le 2^j\le h^{-\frac 12}$, which still has 
$\approx |\log h|$ values of $j$, and thus exhibits the same $L^\infty$ growth as $u_h$. This idea does not however work to yield a counterexample for the form $\xi_1\xi_2+x_1^2+x_2^2$.

\end{document}